\documentclass[12pt, a4paper, reqno, oneside]{amsart}
\usepackage{amssymb}  
\usepackage{amsmath} 
\usepackage{amsthm}  
\usepackage{graphicx}
\usepackage{color}

\usepackage[T2A]{fontenc}
\usepackage[utf8]{inputenc}
\usepackage[russian, english]{babel}
\usepackage{cite, verbatim}
\usepackage{multirow}

\newtheorem{theorem}{Theorem}
\newtheorem*{theorem*}{Theorem}
\newtheorem{lemma}{Lemma}
\numberwithin{lemma}{section}
\newtheorem{prop}[lemma]{Proposition}

\theoremstyle{remark}

\numberwithin{equation}{section}

\newcommand{\Aut}{\operatorname{Aut}}

\newcommand{\meo}{\operatorname{meo}}

\newcommand{\pMod}[1]{\,(\mathrm{mod}\ #1)}

\begin{document}

\vspace{1cm}

\title[\textsc{On recognition of simple classical groups}]{\textsc{On recognition of simple classical groups with prime graph independence number $4$ by spectrum}}

\author{\textsc{M.A. Grechkoseeva}}
\address{Novosibirsk State University, Pirogova, 1, Novosibirsk 630090, Russia;\newline \hspace*{3mm}  Sobolev Institute of Mathematics, Koptyuga 4, Novosibirsk 630090}
\email{grechkoseeva@gmail.com}

\author{\textsc{V.M. Rodionov}}
\address{Novosibirsk State University, Pirogova, 1, Novosibirsk 630090, Russia}
\email{crowulll@gmail.com}


\begin{abstract}

Let $L$ be one of the finite simple classical groups $L_8(q)$, $U_8(q)$, $O_{10}^+(q)$, $O_{10}^-(q)$ or $O_{12}^+(q)$, with $q$ odd. We prove that every finite group having the same set of element orders as $L$ is an almost simple group with socle isomorphic to $L$. This completes the study of the recognition-by-spectrum problem for simple classical groups whose prime graph independence number is equal to $4$.

\end{abstract}

\begingroup
\def\uppercasenonmath#1{} 
\let\MakeUppercase\relax 
\maketitle
\endgroup

\section*{Introduction}

Given a finite group $G$, we write $\omega(G)$ for the set of element orders of $G$ and refer to this set as the \textit{spectrum} of $G$. Groups with equal spectrum are said to be \textit{isospectral}. Let $\mathcal{H}(G)$ denote the set of isomorphism classes of finite groups isospectral to $G$ and set 
$h(G)=|\mathcal{H}(G)|$. We say that the \textit{problem of recognition by spectrum} is solved for $G$ if either it is proved that  $h(G)=\infty$, or the set $\mathcal{H}(G)$ is found explicitly. Observe that $h(G)=\infty$ whenever $G$ has a nontrivial normal solvable subgroup \cite[Lemma 1]{98Maz.t}, so the problem of recognition by spectrum is of interest only for groups with trivial solvable radical. 
The most basic groups having such structure are nonabelian simple groups. 

At present, the problem of recognition by spectrum is solved for all nonabelian simple groups except some classical groups (see \cite{23Survey} for a survey). These remaining groups can be shortly described in terms of the prime graph. The \textit{prime graph} of a group $G$ is the graph $GK(G)$ whose vertices are prime divisors of the order of $G$ and two vertices $r$ and $s$ are adjacent if and only if $r\neq s$ и $rs\in \omega(G)$. The \textit{independence number} $t(G)$ is the maximum number of pairwise nonadjacent vertices of $GK(G)$.
Adjacency conditions for the prime graphs of finite nonabelian simple groups, as well as the independence numbers of these graphs, are determined in  \cite{05VasVd.t, 11VasVd.t}. Also it is worth noting that the whole spectra of finite simple groups are known (see the introduction in \cite{18But.t}).

Now we return to simple groups with unsolved recognition problem. It follows from  \cite[Theorem 2.1]{23Survey}, \cite{24GrPan.t}, \cite{25Gr.t} and \cite{25Pan_arxiv} that as of 2025, the problem of recognition by spectrum is solved for all nonabelian simple groups except the groups $L$ satisfying the following conditions:

\begin{enumerate}
\item $L$ is a simple classical group over a field of odd characteristic;
  \item $4\leq t(L)\leq 13$;
 \item $GK(L)$ is connected.
\end{enumerate}

The case $t(L)\geq 5$ is studied in \cite{26Sta.t}. Our goal is to handle the case $t(L)=4$, that is, to prove the following theorem.

\begin{theorem}\label{t:t4}
Suppose that $L$ is a finite simple classical group over a field of odd characteristic and $t(L)=4$. Then the problem of recognition by spectrum is solved for~$L$.
\end{theorem}

By  \cite[Tables 2 and 3]{11VasVd.t}, the simple classical groups $L$ in odd characteristic with $t(L)=4$ are in the following series:  $L_3^\pm(q)$, $L_7^\pm(q)$, $L_8^\pm(q)$, $S_8(q)$, $O_9(q)$, $O_8^-(q)$, $O_{10}^\pm(q)$, and $O_{12}^+(q)$ (we denote simple classical groups mostly according to \cite{85Atlas} but also write $L_n^+(q)$ for $L_n(q)$ and $L_n^-(q)$ for $U_n(q)$). The groups $L_3^\pm(q)$, $L_7^\pm(q)$, $S_8(q)$, $O_9(q)$, and $O_8^-(q)$ have disconnected prime graphs for all $q$ by \cite{81Wil}. So we are left with the groups $L_8^\pm(q)$,  $O_{10}^\pm(q)$ and $O_{12}^+(q)$. 

Let $L$ be an arbitrary nonabelian simple group, and suppose that $\mathcal H(L)\subseteq \mathcal{AS}(L)$, where $\mathcal{AS}(L)$ is the set of isomorphism classes of almost simple groups with socle $L$. Then $h(L)<\infty$. Furthermore, since the set $\mathcal H(L)\cap \mathcal{AS}(L)$ is explicitly known (see \cite[Theorem 3.7]{23Survey}), the set $\mathcal H(L)$ is explicitly known as well, thereby the problem of recognition by spectrum is solved for $L$. 

These two observations  show that Theorem \ref{t:t4} will follow from the next one. 

\begin{theorem}\label{t:cor}
Let $L$ be one of the simple groups  $L_8^\pm(q)$,  $O_{10}^\pm(q)$, and $O_{12}^+(q)$, with $q$ odd. Suppose that $G$ is a finite group such that $\omega(G)=\omega(L)$. Then $G$ is an almost simple group with socle isomorphic to $L$. 
\end{theorem}

Assume that $L$ is a simple group as in the hypothesis of Theorem \ref{t:cor},  $\omega(G)=\omega(L)$ and $G$ is not an almost simple group with socle isomorphic to $L$. By a series of previous results (see Lemma \ref{l:red} in Section \ref{s:prel}), this implies that some nonabelian composition factor of $G$ is a classical group in characteristic other than the defining characteristic of $L$. So Theorem \ref{t:cor} is a corollary of the following result. 

\begin{theorem} \label{t:main}
Let $L$ be one of the simple groups  $L_8^\pm(q)$,  $O_{10}^\pm(q)$, and $O_{12}^+(q)$, with $q$ odd. Suppose that $G$ is a finite group such that $\omega(G)=\omega(L)$. If $S$ is a nonabelian composition factor of $G$ and $S$ is a classical group, then the defining characteristic of $S$ divides $q$.
\end{theorem}

The paper is organized as follows. Section~\ref{s:prel} contains notation, definitions and some generic facts concerning the structure of groups isospectral to simple groups. In Section~\ref{s:graph} we collect some calculations and some properties of the primes graphs and spectra of classical group $L$ with $t(L)=4$. The proof of Theorem \ref{t:main} is divided into two cases, depending on whether $t(S)>4$ or not. These cases are considered in Sections \ref{s:large} and \ref{s:t4} respectively. Finally, at the end of Section \ref{s:t4}, we explain how to find $h(L)$ and $\mathcal{H}(L)$ (if $h(L)<\infty$) for a group $L$ of Theorem \ref{t:t4} by using the information in \cite{23Survey}.

\section{Notation and some general results} \label{s:prel}

Let $a$ and $b$ be nonzero integers, $r$ a prime. The greatest common divisor of $a$ and $b$ is denoted by $(a, b)$. The set of prime divisors of $a$ is denoted by $\pi(a)$. We write $(a)_r$ for the $r$-part of $a$, that is, the highest power of  $r$ dividing $a$ and put $(a)_{r'}=a/(a)_r$. If $r$ is odd and coprime to $a$, then $e(r, a)$ is the multiplicative order of $a$ modulo $r$. If $a$ is odd, then $e(2, a) = 1$ if $a\equiv 1\pMod 4$ and $e(2, a) =2$ if $a\equiv 3\pMod 4$. The following lemma is well known. 

\begin{lemma}\label{l:r-part}
Suppose that $a$ is an integer, $|a|>1$, and $m$ is a positive integer. 
\begin{enumerate}
\item  If $r$ is an odd prime and  $a\equiv 1\pMod r$, then $(a^m-1)_r=(m)_r(a-1)_r$.
\item  If $a\equiv 1\pMod 4$ or $m$ is odd, then $(a^m-1)_2=(m)_2(a-1)_2$.
\item If $a\equiv -1\pMod 4$ and $m$ is even, then $(a^m-1)_2=(m)_2(a+1)_2$.
\end{enumerate}
\end{lemma}

Let $|a|>1$ and $i\geq 1$. A prime $r$ is a \textit{primitive prime divisor} of $a^i-1$ if $e(r, a) = i$. The set of all primitive prime divisors of $a^i-1$ is denoted by  $R_i(a)$, and writing $r_i(a)$, we mean an arbitrary primitive divisor.  Bang \cite{86Bang} and Zsigmondy \cite{Zs} proved that primitive divisors exist for  almost all pairs $a$ and $i$. 

\begin{lemma}[Bang--Zsigmondy]
Suppose that $a$ is an integer, $|a|>1$. For every positive integer $i$, the set $R_i(a)$ is not empty, excepting the cases when $a=2$ and $i\in\{1,6\}$, or  $a=-2$ and $i\in\{2,3\}$, or $a=3$ and $i=1$, or $a=-3$ and $i=2$.
\end{lemma}

For $i\geq 3$, we write $k_i(a)$ to denote the product of $(a^i-1)_r$ over all $r\in R_i(a)$. As usual, $\varphi(x)$ and $\Phi_i(x)$ denote the Euler's totient  function and the $i$th cyclotomic polynomial respectively. It is well known that  $\Phi_i(-x)=\Phi_{2i}(x)$ if $i\geq 3$ is odd and $\Phi_i(-x)=\Phi_i(x)$ if $i$ is divisible by $4$. Also $\Phi_i(a)>0$ for all $|a|>1$ and $i\geq 3$ 

\begin{lemma}\label{l:kiPhi}
Let $a$ and $i$ be integers, $|a|>1$ and $i\geq 3$. 
Let $r$ be the largest prime divisor of $i$ and put $l = (i)_{r'}$. Then
$$k_i(a) = \dfrac{\Phi_i(a)}{(r,\Phi_l(a))}.$$
Furthermore, if $l$ does not divide $r-1$, then $(r, \Phi_l(a)) = 1$.
\end{lemma}
\begin{proof}
See~\cite[Proposition 2]{97Roi}. 
\end{proof}

We write $\meo(G)$ for the maximum element of $\omega(G)$ and $\mu(G)$ for the subset of  $\omega(G)$ consisting of numbers that are maximal under divisibility. The exponent of $G$, that is, the least common multiple of elements of $\omega(G)$, is denoted by $\exp(G)$. If $r$ is a prime, then $\exp_r(G)=(\exp(G))_r$ and  $\exp_{r'}(G)=(\exp(G))_{r'}$. As usual, $\pi(G)=\pi(|G|)$. If $r\in\pi(G)$, then $t(r,G)$ is the maximum size of a coclique in $GK(G)$ containing $r$. 

As we mentioned in the introduction, the structure of prime graphs and the spectra of simple groups are known. For a simple classical group $L$ in characteristic $p$, we take the numbers $t(2,L)$ and $t(p,L)$ from \cite{05VasVd.t}, while the number $t(L)$ is taken from \cite{11VasVd.t}, where some corrections to results of \cite{05VasVd.t} are made. Our standard reference for the spectra of linear and unitary groups is \cite{08But.t}, and for the spectra of symplectic and orthogonal groups is \cite{10But.t} (with corrections from \cite[Lemma 2.3]{16Gr.t}).

If $L$ is written as $L_n^\pm(q)$, $S_{2n}(q)$, $O_{2n+1}(q)$, or $O_{2n}^\pm(q)$, then we say that $L$ is a group over a field of order $q$ (despite the fact that $SU_n(q)$ is a subgroup of $SL_n(q^2)$). By the untwisted Lie rank of $L$ we mean the rank of the untwisted root system of $L$ (in the notation of the preceding sentence, it is equal to $n-1$ if $L=L_n^\pm(q)$ and $n$ otherwise). 

\begin{lemma}\label{l:meo}
If $L$ is a simple classical group of untwisted Lie rank $l$ over a field of order~$q$, then $\meo(L)<q^{l+1}/q\leq 2q^l$.
\end{lemma}

\begin{proof}
 See \cite[Lemma 1.3]{09VasGrMaz1.t}. 
\end{proof}

The rest of the section is concerned with the structure of  finite groups isospectral to simple classical groups.

\begin{lemma}\label{l:str}
Let $L$ be a finite simple group such that $t(2,L)\geq 2$ and $t(L)\geq3$ and suppose that $G$ is a finite group such that $\omega(G)=\omega(L)$. Then the following hold. 

\begin{enumerate}
\item If $K$ is the solvable radical of $G$, then $S\leq \overline G=G/K\leq \Aut S,$ where $S$ is a nonabelian simple group.

\item If $\rho$ is a coclique in $GK(G)$ and $|\rho|\geq 3$, then at most one prime of $\rho$ divides $|K|\cdot|\overline{G}/S|$. In particular, $t(S)\geq t(L)-1$.

\item If $r\in\pi(G)$ is not adjacent to $2$ in $GK(G)$, then $r$ is coprime to $|K|\cdot|\overline{G}/S|$. In particular, $t(2,S)\geq t(2,L)$.

\item $K$ is nilpotent.
\end{enumerate}

\end{lemma}

\begin{proof}  See \cite{05Vas.t, 09VasGor.t} for (1)--(3) and \cite{20YanGrVas} for (4).
\end{proof}

Observe that by \cite[Theorem 7.1]{05VasVd.t}, the condition $t(2,L)\geq 2$ always holds if $L$ is a~simple group of Lie type. 

\begin{lemma}\label{l:red} Let $L$ be a simple classical group over a field of odd characteristic $p$, $t(L)\geq 4$ and $L\neq S_8(q)$, $O_9(q)$. Suppose  $G$ is a finite group such that $\omega(G)=\omega(L)$ and let $K$ and $S$ be as in Lemma \ref{l:str}. Then either  
$S\simeq L$ and $K=1$, or $S$ is a classical group over a~field of characteristic other than $p$ and $S\neq L_2(u)$.
\end{lemma}

\begin{proof} 
If $S\simeq L$, then $K=1$ by \cite[Theorem 3.6]{23Survey}.  If  $S\not\simeq L$, then $S$ is a group of Lie type over a field of characteristic other than $p$ by \cite[Theorem 3.8]{23Survey}. Note that exceptional cases of two preceding theorems are eliminated by the assumption $t(L)\geq 4$. Now $S$ is not exceptional by \cite[Theorem 1]{24GrPan.t} and $S\not\simeq L_2(u)$ by \cite[Proposition 3.8]{20GrZv.t}. 
\end{proof}

\begin{lemma}\label{l:v_ind}
Suppose that $K$ is a normal nilpotent subgroup of a finite group $G$ and $S\leq \overline G= G/K \leq \Aut S$, where $S$ is a simple classical group over a field of characteristic $v$ and $S\neq L_2(v^\alpha)$. 
\begin{enumerate}
 \item If $r\in\pi(K)\cup\pi(\overline G/S)\setminus\{v\}$, then $rv\in\omega(G)$. In particular, $t(v,G)\leq t(v,S)$.
 \item If $v\in\pi(K)$, then $t(v,G)\leq 3$. Furthermore, if $t(v,G)=3$, then the untwisted Lie rank of $S$ is odd. 
\end{enumerate}

\end{lemma}

\begin{proof}

(1) Let $r\in\pi(\overline G/S)$. Suppose that either $r\in\{2,3\}$ or $S=L_m^\tau(u)$ and $r$ divides $(m,u-\tau)$. Applying \cite[Proposition 3.1]{05VasVd.t}, we see that $rv\in\omega(S)$ unless $S=L_3^\tau(u)$, $r=3$ and $(u-\tau)_3=3$. In this exceptional case $u$ is not a cube (otherwise $(u-\tau)_3=3(u^{1/3}-\tau)_3>3$ by Lemma \ref{l:r-part}) and so $PGL_3^\tau(u)\leq \overline G$. It remains to note that $PGL_3^\tau(u)$ has an element of order  $v(u-\tau)$. Thus we may assume that $\overline G$ contains a field automorphism of $S$ of order $r$, and then $rv\in\omega(G)$ since the centralizer of any field automorphism has an element of order $v$.  

Let $r\in\pi(K)$. Since $K$ is nilpotent, we may assume that $K$ is an $r$-group. For $S\neq L_2(v^\alpha)$, there is an elementary abelian subgroup $V$ of order $v^2$ in $S$. The preimage of $V$ in $G$  contains an element of order $vr$ by \cite[Theorem 5.3.14]{68Gor}.

(2) As above, we may assume that $K$ is a $v$-group.  Suppose that $\rho$ is a coclique of  $GK(G)$ containing $v$ such that $|\rho|=t(v,G)$. Then $\rho\subseteq \pi(S)$ by (1). Now we use the information about cocliques of $GK(S)$ containing $v$ from \cite[Propositions 6.3 and 6.4]{05VasVd.t}.

If $S=S_{2m}(u)$ or $O_{2m+1}(u)$, then $t(v,S)=2$ if $m$ is even and $t(v,S)=3$ if $m$ is odd, so in this case we are done. 

Let $S=L_{m+1}^\tau(u)$, where $m\geq 2$. Then $t(v,S)=3$ and every coclique of  $GK(S)$ of size 3 which contains $v$ is of the form $\{v,r_{m}(\tau u), r_{m+1}(\tau u)\}$. If $m$ is even, then $vr_{m}(\tau u)\in \omega(G)$ by  \cite[Corollary 9(2)]{11Zav}. 

Let $S=O_{2m}^\pm(u)$. If $m$ is odd, then $t(v,S)=3$. If $m$ is even, then $t(v,S)\leq 4$ and every coclique of size $t(v,S)$ containing $v$ is a subset of  $\{v, r_{2m}(u), r_{2m-2}(u), r_{m-1}(u)\}$. By \cite[Corollary 19]{07SupZal}, it follows that $vr_{2m-2}(u)$, $vr_{m-1}(u)\in\omega(G)$, and this completes the proof.
\end{proof}

\section{Preliminary calculations and some properties of prime graphs and spectra}\label{s:graph}

In this section we collect preliminary calculations involving numbers of the form $k_i(a)$ and some properties of the prime graphs and spectra of  classical groups $L$ having connected prime graph and $t(L)=4$.

\begin{lemma} \label{l:bounds} Let $q$ be a prime power and $\varepsilon\in\{+,-\}$. 
\begin{enumerate}
\item If $i\in \mathbb N$ and $\varphi(i)>4$, then $k_i(\varepsilon q)>q^4$.
\item $k_{12}(q)$, $k_8(q)\geq (q^4+1)/2$.
\item $k_5(q)k_{10}(q)>q^8/5+1$ and if $\varepsilon q\neq -4$, then $k_5(\varepsilon q)>q^4/6+1.$
\item $k_7(\varepsilon q)>q^6/8+1$. 
\item For $m_8(\varepsilon q)=(q^8-1)/((8,q-\varepsilon)(q-\varepsilon))$, we have $m_8(\varepsilon q)>3a^7/28$ and if $q$ is odd and $\varepsilon q\neq 3$, then $m_8(\varepsilon q)<7a^7/12$.
\end{enumerate}
\end{lemma}

\begin{proof}
(1) If $\varphi(i)\geq 8$, then $k_i(\varepsilon q)>q^{\varphi(i)/2}\geq q^4$ by \cite[Lemma 1.5]{15Vas}. If $\varphi(i)=6$, then $i=7,9,14, 18$, and $k_i(\varepsilon q)\geq (q^7+1)/(7a+7)>q^4$ for all $q\geq 4$. If $q=2, 3$, then $k_i(\varepsilon q)\geq (q^6-q^3+1)/3>q^4$.

(2) This follows since $k_{12}(q)=q^4-q^2+1$ and $k_8(q)=(q^4+1)/(2,q-1)$.

(3) If $\varepsilon=+$, then $5(k_5(\varepsilon q)-1)=(q^5-1)/(q-1)-5>q^4$. 
In particular, $k_5(q)k_{10}(q)=k_5(q^2)>q^8/5+1$. If $(5,q-\varepsilon)=1$, then $k_5(\varepsilon q)-1>q^4/2$. Finally, if $\varepsilon=-$ and $(5,q-\varepsilon)=5$, then $q\geq 9$ and $$5(k_5(\varepsilon q)-1)=\frac{q^5+1}{q+1}-5>q^4-q^3=\frac{q-1}{q}\cdot q^4\geq \frac{8q^4}{9}>\frac{5q^4}{6}.$$

(4) If $(7,q-\varepsilon)=1$ or $\varepsilon=+$, then $k_7(\varepsilon q)>q^6/2$. Otherwise, $q\geq 13$ and $$7(k_7(\varepsilon q)-1)=\frac{q^7+1}{q+1}-7\geq q^6-q^5\geq \frac{12q^6}{13}>\frac{7q^6}{8}.$$

(5) The first inequality is proved by the same argument as (4). If $(8,q-\varepsilon)\geq 4$ or $\varepsilon=-$, then
$(q^8-1)((8,q-\varepsilon)(q-\varepsilon))<q^7/2.$ If $q\geq 7$, then $$\frac{q^8-1}{(8,q-\varepsilon)(q-\varepsilon)}<\frac{q^8}{2(q-1)}\leq \frac{7q^7}{12}.$$ This left us with the case $\varepsilon=+$, $(q-1)_2=2$ and $q<7$, whence $q=3$.
\end{proof}

\begin{lemma}\label{l:ineq}
Let $q$ be a prime power, $\varepsilon\in\{+,-\}$,  $I\subseteq \mathbb N$, $|I|\geq 2$ and suppose that $\varphi(i)>2$ for all $i\in I$. Then $\prod_{i\in I}k_i(\varepsilon q)\geq Cq^{4|I|}+1/2$, where $C=1/12$ if $|I|=2$ and $C=1/24$ if $|I|\geq 3$. 
\end{lemma}

\begin{proof}
 The claim follows from (1)--(3) of Lemma \ref{l:bounds} and the inequality $$(x+1/2)(y+1/2)>xy+1/2 \text{ for }x,y\geq 1,$$ unless we have $\varepsilon q=4$,  $5\in I$ and $10\not\in I$. In this case we take $i\in I\setminus\{5\}$. Then either $k_i(\varepsilon q)\geq q^4+1$ or $i=12$, and so  $k_5(\varepsilon q)k_i(\varepsilon q)\geq k_5(\varepsilon q)k_{12}(\varepsilon q)=41\cdot 241>4^8/12+1/2$. 
\end{proof}

\begin{lemma}\label{l:k5}
Suppose that $a$, $b\in \mathbb Z$,  $|a|$ and $|b|$ are prime powers, $|a|\neq |b|$ and $k_5(a)=k_5(b)$. Then $(5,a-1)\neq (5,b-1)$ and $k_5(-a)\neq k_5(-b)$.
\end{lemma}

\begin{proof}
Assume that $(5,a-1)=(5,b-1)$. Then $(a^5-1)/(a-1)=(b^5-1)/(b-1)$. Since  $f(x)=(x^5-1)/(x-1)$ is strictly increasing on $[2,+\infty]$ and strictly decreasing on $[-\infty,-2]$, we may assume that $a<0<b$. If $|a|\leq b$, then it is easy to see that $f(a)<f(b)$. If $|a|\geq b+1$, then $f(a)>|a|^4-|a|^3\geq b(b+1)^3>f(b)$, a contradiction.

Assume that $k_5(-a)=k_5(-b)$. By the above, it follows that $(5,a-1)\neq (5,b-1)$ and $(5,a+1)\neq (5,b+1)$. We may assume that $(5,a-1)=5$ and $(5,b-1)=1$. Then $(5,a+1)=1$ and $(5,b+1)=5$. Hence $(5,a^2-1)=(5,b^2-1)$ and  $k_5(a^2)=k_5(a)k_5(-a)=k_5(b)k_5(-b)=k_5(b^2)$. This contradicts the first claim since $a^2\neq b^2$. 
\end{proof}

\begin{lemma} \label{l:igcd} Let $a\in \mathbb Z$, $a$ is odd and $|a|>1$. The following hold:
\begin{enumerate}
 \item $(k_8(a)-1, k_7(a)-1)$ divides $3(a+1)$ if $(7,a-1)=1$ and $75(a-1)$ otherwise;
 
 \item $(k_8(a)-1, k_5(a)-1)$ divides $(a^2+1)(a+1)$ if $(5,a-1)=1$ and $4(a-1)$ otherwise;
 
 \item $(k_5(a)-1, k_{10}(a)-1)$ divides $2a(a^2+1)$ if $(5,a^2-1)=1$ and $4(a-1)$ otherwise.
 
\end{enumerate}

\end{lemma}

\begin{proof}

If $(7,a-1)=1$, then $$(k_8(a)-1, k_7(a)-1)=\left(\frac{a^4-1}{2}, \frac{a(a^6-1)}{a-1}\right)=$$ $$=(a^2-1, (a+1)(a^4+a^2+1))=(a+1)(a-1, a^4+a^2+1)=(a+1)(a+1,3).$$

If $(7,a-1)=7$, then $(k_8(a)-1, k_7(a)-1)$ divides $(f(a), g(a))$, where $f(x)=x^4-1$ and $g(x)=(x^7-7x+6)/(x-1)$. Let  $h(x)\in \mathbb Z[x]$ be the primitive greatest common divisor of $f(x)$ and $g(x)$ (regarded as polynomials over $\mathbb Q$). There are $f_1(x),g_1(x)\in \mathbb Q[x]$ such that 
$f(x)f_1(x)+g(x)g_1(x)=h(x)$. If $m$ is the least positive integer such that $mf_1, mg_1\in \mathbb Z[x]$, then $(f(a), g(a))$ divides $mh(a)$ for all integer $a$. Using, for example, the commands \texttt{Gcd} and \texttt{ GcdRepresentation} of GAP \cite{GAP}, one can verify that $h(x)=x-1$ and $m=75$. Thus (1) follows. A similar calculation yields (2) and (3). 
\end{proof}

The next two lemmas  are proved in the same way as Lemma \ref{l:igcd}. 

\begin{lemma} \label{l:igcd1} Let $a\in \mathbb Z$ and $|a|>1$. 
\begin{enumerate}
 \item If $a$ is even, then $(7k_8(a)-1, \Phi_i(a))$ divides $31\cdot 61$, $43$ or $13$ for $i=5$, $3$ or $4$, respectively.
 
 \item If $a$ is odd, then $(7k_8(a)-1, \Phi_i(a))$ divides $11\cdot 151$, $39$ or $12$ for $i=5$, $3$ or $4$, respectively. 
 
 \item Suppose that $(5,a-1)=5$. Then $(7k_5(a)-1, \Phi_i(a^2))$ divides $2\cdot 1297$ or $3\cdot 13\cdot 109$, and $(k_5(a)-1, \Phi_i(a^2))$ divides $2\cdot 97$ or $3\cdot 7\cdot 31$ for $i=4$ or $3$, respectively.
   
 \item If $(5,a-1)=1$, then $(7k_5(a)-1, \Phi_i(a^2))$ divides $2\cdot 337$ or $3\cdot 19\cdot 43$ for $i=4$ or $3$, respectively.
\end{enumerate}
\end{lemma}

\begin{lemma}\label{l:igcd2} Let  $a\in \mathbb Z$, $|a|>1$ and $f(x)=x^3+2x^2+3x+4$. Then $(f(a),\Phi_i(a^2))$ divides $10$,
$4$, $7\cdot 31$ or $2\cdot 97$ for $i=1,2,3$ or $4$, respectively. 
\end{lemma}

In the rest of this section,  $q$ is a power of a prime $p\geq 2$ and $L$ is one of the groups $L_8^\varepsilon(q)$, $O_{10}^\varepsilon (q)$, or $O_{12}^\varepsilon(q)$, where $\varepsilon\in \{+,-\}$ in the first two cases and $\varepsilon=+$ in the last case. We introduce the numbers $z$, $y$, $m_z(L)$ and $m_y(L)$ according to Table \ref{tab:zy}. Also for $x\in \{z,y\}$, we define $R_x(L)=R_x(\varepsilon q)$,  $r_x(L)=r_x(\varepsilon q)$ and $k_x(L)=k_x(\varepsilon q)$. The caption of Table \ref{tab:zy} is clarified by (2) and (3) of Lemma \ref{l:zy} below.

\begin{table}[th]
\caption{The primes not adjacent to $2$ and $p$ in $GK(L)$}\label{tab:zy}
\renewcommand{\arraystretch}{1.3}
$
\begin{array}{|l|l|c|c|c|c|}
\hline
L&\text{ Conditions on }L &z&y&m_z(L) & m_y(L)\\
\hline
L_8^\varepsilon(q)& 8<(q-\varepsilon)_2& 8&7& \frac{q^8-1}{8(q-\varepsilon)}&
\frac{q^{7}-\varepsilon}{8}\\
&8\geq (q-\varepsilon)_2&7& 8&\frac{q^{7}-\varepsilon}{(8,q-\varepsilon)}&\frac{q^8-1}{(q-\varepsilon)(8,q-\varepsilon)}\\
\hline
O_{10}^\varepsilon(q)   &  q\equiv \varepsilon\,(\mathrm{mod}\ 4) &{8}& 5&\frac{(q^{4}+1)(q+\varepsilon)}{4}& \frac{q^5-\varepsilon}{4}\\
 &  q\not\equiv \varepsilon\,(\mathrm{mod}\ 4) & 5 & 8& \frac{q^5-\varepsilon}{(2,q-\varepsilon)}& \frac{(q^{4}+1)(q+\varepsilon)}{(2,q-\varepsilon)}\\
\hline
O_{12}^+(q)&q\equiv1\,(\mathrm{mod}\ 4)& 10& 5&\frac{q^{5}+1}{2}& \frac{q^{5}-1}{2}\\
&q\not\equiv 1\,(\mathrm{mod}\ 4)& 5& 10&\frac{q^{5}-1}{(2,q-1)}& \frac{q^{5}+1}{(2,q-1)}\\
\hline
\end{array}
$
\end{table}
 
\begin{lemma} \label{l:zy} Let $r\in R_z(L)$ and $s\in R_y(L)$. The following hold.
\begin{enumerate}
  \item $t(r,L)=t(s,L)=t(L)$.
 \item $\{2,r\}$ is a coclique in $GK(L)$ and $m_z(L)$ is the only number in $\mu(L)$ divisible by $r$.
 \item $\{p,r,s\}$ is a coclique in $GK(L)$ and $m_y(L)$ is the only number in $\mu(L)$ divisible by $s$.
 \item $m_z(L)$ and $m_y(L)$ are coprime, in particular, $r$ and $s$ have no common neighbors in $GK(L)$.
 \item If $GK(L)$ is connected, then, conversely, two different primes $r_1$ and $s_1$ are not adjacent and have no common neighbors in $GK(L)$ only if one of them lies in $R_z(L)$ and another lies in $R_y(L)$.
 \item For every $x\in\{z,y\}$, there is $\eta_x\in\{+1,-1\}$ such that if $r\in\pi(m_x(L))\setminus R_x(L)$ and $t(r,L)=2$, then $r\in\pi(q-\eta_x)$.
 \end{enumerate}
\end{lemma}

\begin{proof}
(1)--(3) The claims concerning cocliques follow straightforwardly from  \cite{05VasVd.t, 11VasVd.t}. The claims on $m_z(L)$ and $m_y(L)$ can be easily verified by using the description of $\omega(L)$ in \cite{08But.t,10But.t} and the Bang--Zsigmondy lemma. 

(4) A direct calculation using Lemma \ref{l:r-part}.

(5) Let $L=L_8^\varepsilon(q)$. By assumption, $GK(L)$ is connected, so there is $r_0\in\pi(L)\setminus R_7(q)$ adjacent to some $r_7(q)$. By  \cite[Propositions 2.1, 2.2, 3.1, 4.1 and 4.2]{05VasVd.t}, it follows that either $r_0\in R_1(\varepsilon q)\setminus\{2\}$, or  $(q-\varepsilon)_2>8$ and $r_0=2$. In either case $r_0$ is adjacent to every prime in $\pi(L)\setminus R_8(q)$. Thus one of $r$ and $s$, say $r$, lies in $R_8(\varepsilon q)$.  On the other hand, the neighborhood of any $r_8(q)$  includes  $R_2(\varepsilon q)\setminus\{2\}$, and if $1<(q-\varepsilon)_2<8$, then it includes $2$ as well. If $s_0\in R_2(\varepsilon q)\setminus\{2\}$, or $s_0=2$ and $1<(q-\varepsilon)_2<8$, then $s_0$ is adjacent to every prime in $\pi(L)\setminus R_7(\varepsilon q)$ and, therefore, $s\in R_7(\varepsilon q)$. It remains to consider the case when $\pi(q+\varepsilon)\subseteq\{2\}$ and $(q-\varepsilon)_2\geq 8$. In this case $q=2^l-\varepsilon$ and $(q-\varepsilon)_2=(2^l-2\varepsilon)_2\leq 4$, a contradiction.  
 
Let $L=O_{10}^\varepsilon(q)$. The proof is similar. As above, there is a prime $r_0$ ($s_0$) in $\pi(L)$ adjacent to $r_5(\varepsilon q)$ ($r_8(q)$). It follows that    $r_0\in R_1(\varepsilon q)\setminus\{2\}$, or $(q-\varepsilon)_2>4$ and $r_0=2$. In either case, $r_0$ is adjacent to every prime in $\pi(L)\setminus R_8(q)$. Also $s_0\in R_2(\varepsilon q)\setminus\{2\}$, or $(q-\varepsilon)_2=2$ and $s_0=2$, and hence $s_0$ is adjacent to every prime in $\pi(L)\setminus R_5(\varepsilon q)$.
 
Let $L=O_{12}^+(q)$ and $\epsilon\in\{+,-\}$. The neighborhood of $r_{5}(\epsilon q)$ is contained in $R_{5}(\epsilon q)\cup R_1(\epsilon q)$. Every element of $R_1(\epsilon q)$ is adjacent to every prime in $\pi(L)\setminus R_{5}(-\epsilon q)$.

(6) It is easy to see that $\pi(m_x(L))\setminus R_x(L)$ is contained in $\pi(q-1)$ or $\pi(q+1)$ unless $L=L_8^\varepsilon(q)$ and $x=8$. In this case  $\pi(m_x(L))\setminus R_x(L)\subseteq\pi(q+\varepsilon)\cup R_4(q)$, and it remains to note that $t(r,L)=4$ for every $r\in R_4(q)$.
\end{proof}

Lemma \ref{l:bounds} yield the following bounds on $k_z(L)$ and $k_y(L)$.

\begin{lemma}\label{l:kzky} If $q$ is odd, then $$\max\{k_z(L), k_y(L)\}\geq (q^4+1)/2, \quad \min\{k_z(L), k_y(L)\}\geq q^4/6+1.$$
\end{lemma}

The next lemma is concerned with the numbers $\exp_r(L)$ for not very small $r$. 

\begin{lemma}\label{l:exp}
Let $q$ be odd, $r\in\pi(L)$ and $r\geq 7$.
\begin{enumerate}
 \item If $r=p$, then $\exp_r(L)=r$ unless $r=7$ and $L\neq O_{10}^\varepsilon(q)$, in which case $\exp_r(L)=r^2$.
 \item If $r\in R_i(\varepsilon q)$, then $\exp_r(L)=(q^i-\varepsilon^i)_r$ unless $r=7$, $i=1$ and $L=L_8^\varepsilon(q)$, in which case $\exp_r(L)=r(q^i-\varepsilon^i)_r$.
 \item If $s\in \pi(L)$, $s>r$, $rs\in\omega(L)$ and $\exp_r(L)\cdot \exp_s(L)\not\in\omega(L)$, then $r=7$, $7\exp_s(L)\in\omega(L)$ and either  $L=L_8^\varepsilon(q)$, $q\equiv \varepsilon\pmod 7$ and $s\not\in R_7(\varepsilon q)$, or $p=7$ and $L\neq O_{10}^\varepsilon(q)$.
\end{enumerate}
\end{lemma}
\begin{proof}
(1) and (2) follow from the description of $\omega(L)$ and Lemma \ref{l:r-part}. 

(3) Since $rs\in\omega(L)$, there is $a\in\mu(L)$ divisible by $rs$. By (1) and (2), this number is divisible by $\exp_s(L)$ and if it is not divisible by $\exp_r(L)$, then $r=7$. Furthermore, either $p=7$ and $L\neq O_{10}^\varepsilon(q)$, or $L=L_8^\varepsilon(q)$ and $q\equiv \varepsilon\pmod 7$. In the letter case, assume that $s\in R_7(\varepsilon q)$. Then $a=m_7(L)=(q^7-\varepsilon)/(8,q-\varepsilon)$ by Lemma \ref{l:zy}, and $(a)_7=7(q-\varepsilon)_7=\exp_7(L)$, a contradiction.
\end{proof}

\section{Proof of Theorem \ref{t:main}: Case $t(S)\neq 4$} \label{s:large}

In this section we reduce the proof of Theorem \ref{t:main} to the case $t(S)=4$.
We begin with general observations which are valid for all $S$ and then confine ourselves to $S$ with $t(S)\geq 5$.

\subsection{General observations} Let $L$ be one of the groups $L_8^\varepsilon(q)$, $O_{10}^\varepsilon (q)$, or $O_{12}^+(q)$, where $q$ is a power of an odd prime $p$. As in Section \ref{s:graph}, we assume that $\varepsilon=+$  if $L=O_{12}^+(q)$. Recall that the theorem follows from \cite[Theorem 1]{25Pan_arxiv} and \cite[Table 7]{23Survey} if $GK(L)$ is disconnected. Thus we may assume that $GK(L)$ is connected. This implies, in particular, that $q\geq 5$ (see \cite{81Wil}).

Suppose that $G$ is a finite group such that $\omega(G)=\omega(L)$ and among nonabelian composition factors of $G$, there is a classical group over a field of characteristic other than $p$. Denote this classical group by $S$ and let $S$ be a group over a field of characteristic $v$ and order $u=v^\alpha$. By Lemma \ref{l:str}, we have $$S\leq \overline G=G/K\leq \Aut S,$$ where $K$ is the solvable radical of $G$, and $K$ is nilpotent.  Also $S\neq L_2(u)$ by Lemma \ref{l:red}. 

Let $z,y\in \mathbb N$ and $R_x(L)$, $k_x(L)$, $m_x(L)$ for $x\in\{z,y\}$ be as in Section \ref{s:graph}, the paragraph before Lemma \ref{l:zy}. Note that $R_z(L)\cap (\pi(K)\cup\pi(\overline G/S))=\varnothing$ by Lemma \ref{l:str}(3), in particular, $k_z(L)\in\omega(S)$. Also $R_y(L)\cap \pi(\overline G/S)=\varnothing$ by \cite[Lemma 3.5]{20GrZv.t}.

\begin{lemma}\label{l:K} The following hold.

\begin{enumerate}
\item If $r\in \pi(K)$ and $r\neq v$, then $t(r,L)=2$.
 \item If $t(r,L)=4$, then $r\not\in\pi(K)$. In particular, $R_y(L)\cap\pi(K)=\varnothing$ and so $k_y(L)\in \omega(S)$.
 \item If $\pi(K)\not\subseteq\{v\}$, then there is a unique $x\in \{z,y\}$ such that $r_x(L)$ divides the order of a proper parabolic subgroup of $S$ and  $\pi(K)\setminus\{v\}\subseteq \pi(q-\eta_x)$, where $\eta_x\in\{+1,-1\}$ is as in Lemma \ref{l:zy}. 
 \item If $r\in\pi(K)$, $(r,u(u^2-1))=1$ and $r^k$ divides the order of a proper parabolic subgroup of $S$, then $r^{k+1}\in\omega(G)$.
\end{enumerate}
\end{lemma}

\begin{proof}
(1) Assume that $\{r,s,t\}$ is a coclique of size 3 in $GK(G)$. Then $s,t\in\pi(S)$ by Lemma \ref{l:str}(2) and  $v\not\in\{s,t\}$ by Lemma \ref{l:v_ind}(1). If one of $s$ and $t$, say $s$, divides the order of a proper parabolic subgroup of $S$, then $rs\in\omega(G)$ by \cite[Lemma 2.16]{20YanGrVas}. If neither of $s$ and $t$ divides the order of a proper parabolic subgroup, then $S=U_3(u)$, $(u+1)_3=3$ and $3\in\{s,t\}$ by \cite[Lemma 1.11]{20GrZv.t}. If $p\neq 3$, then $3$ divides $q+1$ or $q-1$ and then $t(3,L)=2$ by the structure of $GK(L)$. If $p=3$, then $r\in R_y(L)$, so we can replace $s$ and $t$ by $r_z(L)$ and $r_i(\varepsilon q)$, where  $i=4$ or~$5$, and then $3\not\in\{s,t\}$.

(2) The first claim follows from (1) and Lemma \ref{l:v_ind}. The second hold since $t(r,L)=4$ and $r\not\in\pi(\overline G/S)$ for every $r\in R_y(L)$.

(3) Let $r_z\in R_z(L)$ and $r_y\in R_y(L)$. Arguing as in the proof of (1) and noting that $r_z, r_y>3$, we conclude that there is  $x\in\{z,y\}$ such that $r_x$ divides the order of a proper parabolic subgroup of $S$ and so $r_xr\in \omega(L)$ for every $r\in\pi(K)\setminus\{v\}$. It follows that $r\in \pi(m_x(L))\setminus R_x(L)$ and $t(r,L)=2$ by (1). Now we apply Lemma \ref{l:zy}(5). Since $r$ cannot be adjacent to both $r_z$ and $r_y$, the number  $x$ is uniquely determined. 

(4) This follows from \cite[Theorem 1.1]{01DiMZal}  (cf. \cite[Lemma 3.5]{15Vas}).
\end{proof}

Using that $k_z(L), k_y(L)\in \omega(S)$ and applying Lemmas \ref{l:kzky} and \ref{l:meo}, we have \begin{equation}\label{e:kum} q^4/2<\max\{k_z(L), k_y(L)\}\leq \meo(S)\leq 2u^m,\end{equation} where $m$ is the untwisted Lie rank of $S$.

\begin{lemma} \label{l:large}

If $t(r,L)=4$ and $r\in\pi(L)\setminus \pi(\overline G/S)$, then
$r\in\pi(S)$ and $t(r,S)\geq 4$. In particular, $t(S)\geq 4$ and $S\neq L_3^\tau(u)$.
 
\end{lemma}

\begin{proof}
According to the description of cocliques of maximal size in $GK(L)$ in \cite[Tables 2 and 3]{11VasVd.t}, there is a set $I\subseteq \mathbb N$ such that $|I|=3$, $1,2\not\in I$ and $\rho=\{r\}\cup\{ r_i(\varepsilon q)\mid i\in I\}$ is a coclique in $GK(L)$ (cf. Table \ref{tab:t4} in Section \ref{s:t4}). By Lemma \ref{l:K}(2), it follows that $\rho\subseteq \pi(\overline G)$ and so $r\in\pi(S)$. If $R_i(\varepsilon q)\cap\pi(S)\neq \varnothing$ for every $i\in I$, then $t(r,S)\geq 4$.

Assume that $R_i(\varepsilon q)\cap\pi(S)=\varnothing$ for some $i\in I$. Then  
$k_i(\varepsilon q)\in \omega(\overline G/S)$ and using the condition $i\neq 1,2$, we conclude that $k_i(\varepsilon q)$ divides $\alpha$. Also $k_i(\varepsilon q)>q^2/4$ for $q\geq 5$. Hence $u\geq v^{k_i(\varepsilon q)}\geq 2^{q^2/4}$. On the other hand, observing that 
$S$ has an element of order $u+1$ and applying Lemma \ref{l:meo}, we see that 
$u+1\leq \meo(S) \leq \meo(L)\leq 2q^7$.  Thus $2^{q^2/4}<2q^7$, whence $q=5,7$, or $9$. Solving the equation $v^{k_i(\varepsilon q)}<2q^7$, we have either $q=5$, $k_i(\varepsilon q)=7,13$ and $v=2,3$, or $q=7$, $k_i(\varepsilon q)=19$ and $v=2$. Since $u-1\in\omega(S)$ and $a=v^{k_i(\varepsilon q)}-1$ divides $u-1$, it follows that $a\in\omega(S)$. It is easy to check that $e(a, q)\geq 42$, which shows that $a\not\in \omega(L)$, a contradiction.

Every $r\in R_z(L)$ satisfies the hypothesis of the first claim, so $t(S)\geq 4$. Furthermore, by the above,  $GK(S)$ includes a coclique of size 4 consisting of primes larger than $3$, and so $S\neq L_3^\tau(u)$ by \cite[Table 2]{11VasVd.t}.
\end{proof}

\begin{lemma}\label{l:adj}
Let $r, s\in\pi(S)\setminus \pi(v(u\pm 1))$, $r,s\not\in\pi(K)$ and $5<r<s$. If $rs\not\in\omega(S)$ and $rs\in\omega(L)$, then $r=7$ and the following hold: 
\begin{enumerate}
\item either $L=L_8^\varepsilon(q)$, $q\equiv \varepsilon\pmod 7$ and $s\not\in R_7(\varepsilon q)$, or  $p=7$, $L\neq O_{10}^\varepsilon(q)$;
\item $7$ divides $|\overline G/S|$.
 
\end{enumerate}
\end{lemma}

\begin{proof} 
Suppose that $\exp_r(L)\cdot\exp_s(L)\in\omega(L)$. Then $\exp_r(L)\cdot\exp_s(L)\in\omega(\overline G)$ and since  $rs\not\in\omega(S)$, one of  $\exp_r(L)$ and $\exp_s(L)$ lies in $\omega(\overline G/S)$. Denote this number by $a$. As $r,s>5$ and $r,s$ do not divide $u\pm 1$, it follows that $a$ divides $\alpha$. However $(\alpha)_t<\exp_t(S)\leq\exp_t(L)$ for every $t\in\pi(S)\setminus \pi(v(u\pm 1))$ by \cite[Lemma 5.6]{20GrZv.t}, a contradiction.

Thus $\exp_r(L)\cdot\exp_s(L)\not\in\omega(L)$. Lemma \ref{l:exp} implies that  $r=7$ and (1) holds. Also it follows that $7\exp_s(L)\in\omega(G)$. If  $7\not\in\pi(\overline G/S)$, then $\exp_s(L)\in \omega(\overline G/S)$ and we derive a contradiction as in the previous paragraph.
\end{proof}

\subsection{Case $t(S)\geq 5$} 

Suppose that $t(S)\geq 5$. By \cite[Tables 2 and 3]{11VasVd.t}, there is a set  $I\subseteq\mathbb N$ of size $t(S)$ such that the following hold: 

\begin{enumerate}
 \item $R_i(u)\neq \varnothing$ for every $i\in I$ and $\{r_i(u)\mid i\in I\}$ is a coclique in $GK(S)$;
 \item $I\cap\{1,2,4\}=\varnothing$, in particular, $r_i(u)\geq 7$ for every  $i\in I$;
 \item if $I$ contains $3$ or $6$, then $t(S)\leq 6$.
 \end{enumerate}
 
The groups $S$ with $t(S)=5,6$ will be considered separately and, for convenience, these groups are listed in Tables \ref{tab:t6} and \ref{tab:t5} together with the corresponding sets $I$. Also Table~\ref{tab:t5} gives a subset  $K(S)$ of $\omega(S)$ with the following property: if $a\in \omega(S)$ and every prime divisor of $a$ is not adjacent to $2$ in $GK(S)$, then $a$ divides some number in $K(S)$. The sets $K(S)$ are easily extracted from \cite[Tables 4 and 6]{05VasVd.t} and the description of the spectra of classical groups. Clearly,  $k_z(L)$ divides some number in $K(S)$.
In both tables we omit the braces when writing sets.

\begin{table}[ht]
\caption{The simple classical groups $S$ with $t(S)=6$}\label{tab:t6}
$\begin{array}{|l|l|}
\hline
S&I\\
\hline
L_{11}(u), u\neq 2 &6, 7,8,9, 10, 11\\
L_{12}(u)&7,8,9, 10, 11, 12\\

\hline 
U_{11}(u)&3, 5, 8, 14, 18, 22\\
U_{12}(u)&5, 8, 12, 14, 18, 22\\

\hline
S_{14}(u), O_{15}(u)&5, 7, 8 ,10, 12, 14\\
\hline
O^+_{14}(u)&3,5, 7, 8, 10, 12\\
O^+_{16}(u)&5, 7, 8, 10, 12, 14\\
\hline
O^-_{14}(u)&5, 6, 8, 10, 12, 14\\
\hline
\end{array}$
\end{table}

\begin{table}[ht]
\caption{The simple classical groups $S$ with $t(S)=5$}\label{tab:t5}
$\begin{array}{|l|l|l|}
\hline
S&I&K(S)\\
\hline
L_9(u), u\neq 2&5,6,7,8,9&k_9(u), k_8(u)\\
L_{10}(u), u\neq 2&6,7,8,9,10&k_{10}(u), k_9(u)\\
L_{11}(2)&7,8,9,10,11&23\cdot 89, 11\\
\hline 
U_9(u)&3,8,10,14,18&k_{18}(u), k_8(u)\\
U_{10}(u)&3,5,8,14,18&k_5(u), k_{18}(u)\\
\hline
S_{10}(u), O_{11}(u), u\neq 2&3,5,6,8,10 &k_5(u), k_{10}(u)\\
S_{12}(u), O_{13}(u)&3,5,8,10,12&k_{12}(u)\\
\hline
O^-_{12}(u)&3,5,8,10,12&k_{12}(u), k_5(u), k_{10}(u)\\
O_{14}^-(2)&5, 8,10,12,14& 43, 13\\
\hline
\end{array}$
\end{table} 
 
Let $J=\{i\in I\mid R_i(u)\setminus \pi(K)\neq \varnothing\}$. Recall that  $\pi(K)\setminus\{v\}\subseteq \pi(q-\eta)$ for some $\eta\in\{+1,-1\}$ by Lemma \ref{l:K}(3). 

\begin{lemma}  \label{l:I}
Let $t(S)\geq 5$ and let $I$ and $J$ be as above.
\begin{enumerate}
 \item If $r\in R_i(u)\cap\pi(K)$ for some $i\in I$, then $(k_i(u))_r$ divides $q-\eta$. In particular, $k_i(u)$ divides $q-\eta$ for every $i\in I\setminus J$.
 
 \item $|I\setminus J|\geq t(S)-5$.
 
 \item If either $R_i(u)\setminus\pi(K)\neq\{7\}$ for any $i\in J$, or $L$ and $G$ do not satisfy the conditions (1) and (2) in Lemma {\ref{l:adj}}, then $|I\setminus J|\geq t(S)-4$.
  
\end{enumerate}
\end{lemma}

\begin{proof} 
(1) Let $i\in I$ and $r\in R_i(u)\cap\pi(K)$. Then $r\geq 7$ and $r$ divides $q-\eta$. If  $\exp_r(L)=(q-\eta)_r$, then $(k_i(u))_r$ divides $q-\eta$ since  $k_i(u)\in\omega(S)$. 

Suppose that $\exp_r(L)>(q-\eta)_r$. Then $r=7$ and $exp_r(L)=r(q-\eta)_r$ by Lemma \ref{l:exp}(2). It follows that $i=3$ or $6$, and so $t(S)\leq 6$. Hence $S$ is one of the groups in Tables \ref{tab:t6} and \ref{tab:t5}. Expecting these tables, we see that $k_i(u)$ divides the order of a proper parabolic subgroup of $S$. Applying Lemma \ref{l:K}(4) yields $r(k_i(u))_r\in\omega(G)$. So $(k_i(u))_r$ divides $q-\eta$ in this case too. 

(3)  Let $r_i^0\in R_i(u)\setminus \pi(K)$ for $i\in J$. By Lemma \ref{l:adj}, it follows that $\{r_i^0\mid i\in J\}$ is a coclique in $GK(L)$. So $|J|\leq 4$, whence $|I\setminus J|\geq t(S)-4$. 

(2) By (3), we may assume that $R_i(u)\setminus\pi(K)=\{7\}$ for some $i\in J$. The rest of the proof is similar to the proof of (3) with $J'=J\setminus\{i\}$ in place of $J$.
\end{proof}

\begin{prop}\label{p:t7}
 $t(S)\leq 7$.
\end{prop}

\begin{proof}
Assume the opposite. Then $3,6\not\in I$ and applying Lemma \ref{l:I}(3), we conclude that $|I\setminus J|\geq t(S)-4\geq 3$. Now applying Lemma \ref{l:I}(1) and Lemma \ref{l:ineq} yields 
 
 $$\frac{u^{4(t(S)-4)}}{24}+\frac{1}{2}\leq \frac{u^{4(|I\setminus J|)}}{24}+\frac{1}{2}\leq\prod_{i\in I\setminus J}k_i(u)\leq (q-\eta)_{2'}\leq \frac{q+1}{2}.$$
 
Let $m$ be the untwisted Lie rank of $S$. By \cite[Table  2 and 3]{11VasVd.t}, we have $t(S)\geq (m+1)/2$ if  $S$ is linear or unitary, and $t(S)\geq (3m-2)/4$ if  $S$ is symplectic or orthogonal. In either case, $m\leq 2t(S)-1$.   Together with \eqref{e:kum}, this yields $q^4/2<2u^m\leq 2u^{2t(S)-1}.$

Thus $u^{16(t(S)-4)}\leq (12q)^4 < 12^4(4u^{2t(S)-1})$, whence $ 12^4\cdot 4>u^{14t(S)-63}\geq u^{35}$, a contradiction.
\end{proof}

\begin{prop}\label{p:t6}
 $t(S)\neq 6$.
\end{prop}

\begin{proof}

Assume the opposite. Let, as above, $m$ be the untwisted Lie rank of $S$. It follows from Lemma \ref{l:I}(2) that $|I\setminus J|\geq 1$. 

Suppose that $|I\setminus J|\geq 2$. Then applying Lemma \ref{l:I}(1) and Lemma \ref{l:ineq} yields 
$$\frac{u^{8}}{12}+\frac{1}{2}\leq \prod_{i\in I\setminus J}k_i(u)\leq (q-\eta)_{2'}\leq \frac{q+1}{2}.$$
On the other hand, Table~\ref{tab:t6} shows that $m\leq 11$, and therefore  $q^4/2<2u^{11}$  by \eqref{e:kum}. Thus $u^{32}\leq (6q)^4<6^4\cdot 4u^{11}$, whence $u^{21}<6^4\cdot 4$, a contradiction. 

Let $|I\setminus J|=\{i\}$. By Lemma \ref{l:I}(3), this implies that $J$ contains $3$ or $6$. In particular, $i\not\in\{3,6\}$. Also it follows that  $m\leq 10$, and so
\begin{equation}\label{e:t6} q^4/2<2u^{10}.\end{equation}  Assume that $k_i(u)\neq k_{10}(4)$. Then by Lemma \ref{l:bounds} $$\frac{u^4}{6}+\frac{1}{2}\leq k_i(u)\leq  (q-\eta)_{2'}\leq \frac{q+1}{2}.$$  Thus $u^{16}\leq (3q)^4<3^4\cdot 4 u^{10}$, whence $u=2$. Now \eqref{e:t6} yields $q\leq 7$ but then $k_i(u)$ cannot divide $q-\eta$, a contradiction.

If $u=4$ and $i=10$, then $k_i(u)=41$ and so $q\geq 81$, which contradicts  \eqref{e:t6}.
\end{proof}

\begin{prop}\label{p:t5}
$t(S)\neq 5$. 
\end{prop}

\begin{proof} 
Assume the opposite. Recall that $k_z(L)$ divides a number in the set $K(S)$ given in the last column of Table \ref{tab:t5}. Also $k_z(L)\geq q^4/6+1$ by Lemma \ref{l:kzky}. 

Since $k_z(L)\geq 2440$ if $q\geq 11$ and $k_z(L)\in \{313, 521, 11\cdot 71, 1181, 1201\}$ if $q=5,7,9$, it follows that $S$ is neither  $L_{11}(2)$ nor $O_{14}^-(2)$. In the remaining cases every element of $K(S)$ is at most  $u^6+u^3+1$, and therefore \begin{equation}\label{e:t51} q^4/6+1\leq k_z(L)\leq u^6+u^3+1.\end{equation}

Suppose that $R_i(u)\subseteq \pi(K)$ for some $i\in I$. Then $k_i(u)$ divides $q-\eta$ by Lemma \ref{l:I}(1). Hence $(u^2-u+1)/3\leq k_i(u)\leq (q-\eta)_{2'}\leq (q+1)/2$, which yields  \begin{equation}\label{e:t52} 2u^2-2u-1\leq 3q.\end{equation} 
It follows that  $(2u^2-2u-1)^4\leq (3q)^4<3^4\cdot 6(u^6+u^3),$  whence $u\leq 7$ and $q\leq 29$. Since $k_i(u)$ divides $q-\eta$, we conclude that
$q\geq 13$ and $k_i(u)\in\{7, 11,13\}$. Then $u=2$ and $q<5$, a contradiction. 

Now let $R_i(u)\not \subseteq \pi(K)$ for any $i\in I$. By Lemma \ref{l:I}(3),  
this is possible only if $R_j(u)\setminus\pi(K)=\{7\}$ for $j\in I\cap\{3,6\}$ and $L$ and $G$ satisfy the conditions (1) and (2) of Lemma \ref{l:adj}; in particular, $7$ divides $\alpha$. Observe that $\exp_7(L)=\exp_7(\overline G)=\exp_7(S)=(k_j(u))_7$, where the second equality holds by \cite[Lemma 3.10]{15Vas}.

Suppose that $p=7$ and $L\neq O_{10}^\pm(q)$. Then   $49=\exp_7(L)=(k_j(u))_7$. There is an element of order $u^6-1$ or $(u^3\pm 1)(u^2+1)/(2,u-1)$ in $S$. So there is $b$ in $\omega(S)$ such that $b$ is divisible by $k_j(u)$ and $b/k_j(u)>u^2+1$. On the other hand, if $b'\in\omega(L)$ is divisible by $49$, then $b'$ divides $49(q+1)$ or $49(q-1)$. Thus $u^2+1\leq q+1$. This inequality is stronger than \eqref{e:t52}, and so $u\leq 7$. This is a contradiction since $7$ divides $\alpha$.  

Suppose that $L=L_8^\varepsilon(q)$ and $7$ divides $q-\varepsilon$. Then  $7(q-\varepsilon)_7=\exp_7(L)=(k_j(u))_7$. Note that $(k_j(u))_{7'}$ divides $q-\eta$. 

Assume that $\eta=\varepsilon$ or $(k_j(u))_{7'}=1$. Then $k_j(u)$ divides $7(q-\varepsilon)$. So in place of \eqref{e:t52}, we have that  $2u^2-2u-19\leq 21q$, and therefore $$(2u^2-2u-19)^4\leq (21q)^4<21^4\cdot 6(u^6+u^3),$$
whence $u\leq 272$. It follows that $u=2^7$, $j=3$ and $(k_j(u))_{7'}=337$. Since $7(k_j(u))_{7'}$ divides $q-\varepsilon$, we see that $q\geq 14\cdot 337-1$. This contradicts \eqref{e:t51} because $u=2^7$.

Now let $R_j(u)\cap \pi(K)\neq\varnothing$ and $\eta=-\varepsilon$. Let  $s\in R_7(\varepsilon q)$. If $s$ divided the order of a proper parabolic subgroup of $S$, then $\eta$ would be equal to $\eta_7=\varepsilon$. Hence $s$ divides the order of an anisotropic torus of $S$, which implies that $s\in R_i(u)$ for $i\in I$ and so $7s\not\in\pi(S)$. Since $7s\in\omega(L)$, this contradicts  Lemma \ref{l:adj}.
\end{proof}

\section{Proof of Theorem \ref{t:main}: Case $t(S)=4$}\label{s:t4}

In this section, we complete the proof of Theorem \ref{t:main}. By Lemma  \ref{l:large} and Propositions \ref{p:t7}-\ref{p:t5}, we are left with the case when $t(S)=4$ and $S\neq L_3^\tau(u)$. Table \ref{tab:t4} lists all simple classical groups satisfying these conditions. For every $S$, except some groups over the field of order $2$, the table also includes the sets $\Theta(S)$ and $\Theta'(S)$ having the following property: the cocliques of $ GK(S)$ of size  $4$ are exactly the sets of the form  $\Theta(S)\cup\{r\}$, where $r\in \Theta'(S)$. In the definitions of $\Theta(S)$ and $\Theta'(S)$,  $r_i$ stands for $r_i(u)$. The set $K(S)$ in the last column has the same meaning as in Table \ref{tab:t5}. As in the previous tables, we omit braces in sets. 

\begin{table}[ht]
\caption{The simple classical groups $S\neq L_3^\tau(u)$ with $t(S)=4$}\label{tab:t4}
$\begin{array}{|l|l|l|l|}
\hline
S&\Theta(S)&\Theta'(S)&K(S)\\
\hline
L_7(u), u\neq 2&r_7,r_6,r_5,r_4&\varnothing&k_7(u), k_6(u)\\
L_8(u), u\neq 2&r_7,r_6,r_5&r_8,r_4&k_8(u), k_7(u)\\
L_9(2)&&&73, 17\\
L_{10}(2)&&&73, 11\\
\hline 
U_7(u)&r_{14},r_3,r_{10},r_4&\varnothing& k_{14}(u), k_3(u)\\
U_8(u)&r_{14},r_3, r_{10}&r_8, r_4&k_8(u), k_{14}(u)\\
\hline
S_{8}(u), O_9(u), u\neq 2&r_{8}, r_6, r_3,r_4 &\varnothing&k_8(u)\\
S_{10}(2)&&&31, 11\\
\hline
O^+_{10}(2)&&&31, 17\\
O^+_{10}(u), u\neq 2&r_{5}, r_8, r_3&r_6,r_4&k_8(u), k_5(u)\\
O^+_{12}(u)&r_{10}, r_{5}, r_8&r_3,r_6&k_5(u), k_{10}(u)\\
\hline
O^-_{8}(u), u\neq 2&r_{8}, r_6, r_3&r_4,v & k_8(u), k_6(u), k_3(u)\\
O^-_{10}(u), u\neq 2&r_{10}, r_8, r_6&r_3,r_4&k_8(u), k_{10}(u)\\
\hline
\end{array}$
\end{table}

\begin{lemma}\label{l:small}
We have $q\geq 7$ and $u\geq 3$, and if $S$ is one the groups  $L^\tau_8(u)$, $O_{10}^\tau(u)$, or $O_{12}^+(u)$, then $GK(S)$ is connected.  
\end{lemma}

\begin{proof}
If $u=2$, then by Table \ref{tab:t4}, the largest element of $K(S)$ is at most  $k_9(2)=73$. It follows that $q^4/6+1\leq k_z(L)\leq 73$. This is a contradiction because $q\geq 5$. Thus $u\geq 3$.

Let $q=5$. Then the inequalities $(u^4+1)/2\leq \meo(S)\leq \meo(L)\leq 2q^7$ yield $u\leq 23$. For every such $u$, we take an element $k(S)$ in $K(S)$ and check whether $k_z(L)$ divides $k(S)$. This results in $k_z(L)=k_5(5)=11\cdot 71$, $u=17$ and $k(S)=k_{10}(u)=11\cdot 71\cdot 101$; a contradiction since  $101\not\in\pi(L)$.

If $S$ is one of the groups $L^\tau_8(u)$, $O_{10}^\tau(u)$, or $O_{12}^+(u)$, where $u\geq 3$, and $GK(S)$ is connected, then either $u=3$, or $S$ is one of $O_{10}^+(5)$, $L_8(5)$, $L_8(9)$, and $U_8(7)$ (see \cite{81Wil, 89Kon.t}). It follows that the largest element of $K(S)$ is at most $k_7(9)$, so $q^4/6+1\leq k_7(9)$, whence $q\leq 31$. Now it is easy to check that for every  $q\leq 31$, the number $k_z(L)$ does not divided any element of $K(S)$.
\end{proof}

\begin{lemma}\label{l:vRy}
$v\not\in R_y(L)$ 
\end{lemma}

\begin{proof}
Assume the opposite. By Lemma \ref{l:v_ind}, every prime lying in $(\pi(K)\cup\pi(\overline G/S))\setminus\{v\}$ is adjacent to $v$ in $GK(G)$ and so  divides $m_y(L)$ by Lemma \ref{l:zy}(3). Also it follows that $v\geq 7$ and $t(v, S)\geq t(v,L)=4$. By Table \ref{tab:t4}, this implies that $S=O_8^-(u)$. Since $v$ is odd, $R_z(L)\subseteq R_8(u)$ by \cite[Table 6]{05VasVd.t}.

The graph $GK(L)$ is connected, so there is $r\in \pi(m_z(L)/k_z(L))$. Both  $r$ and $p$ are coprime to $m_y(L)$ and therefore lie in $\pi(S)$. Furthermore,  $r$ is adjacent to both $r_z(L)$ and $p$ in $GK(L)$, and so it is adjacent to them in $GK(S)$ too. Since $R_8(u)$ is a connected component of  $GK(S)$, it follows that $p\in R_8(u)$ and hence $pr_z(L)\in\omega(S)$, a contradiction. 
\end{proof}

\begin{lemma}\label{l:ij}
There are $i,j\in \mathbb N$ such that $R_z(L)\subseteq R_i(\tau u)$, $R_y(L)\subseteq R_j(\tau u)$ and the following hold:
\begin{enumerate}
 \item $8\in \{i,j\}\subseteq \{8,6,4,3\}$ if $S=S_8(u)$, $O_9(u)$, or $O_8^-(u)$;
 \item $7\in \{i,j\}\subseteq \{7,6,5,4\}$ if $S=L_7^\tau(u)$;
 \item $\{i,j\}$ is equal to $\{8,7\}$, $\{8,5\}$, or $\{10,5\}$, respectively, if $S$ is one of $L_8^\tau(u)$, $O_{10}^\tau(u)$, or $O_{12}^+(u)$.
\end{enumerate}
\end{lemma}

\begin{proof} Let $r_z\in R_z(L)$ and $r_y\in R_y(L)$.
Both $r_z$ and $r_y$ lie in $\pi(S)$ and their neighborhoods in $GK(S)$ are disjoint. So if $S$ is one of $L_8^\tau(u)$, $O_{10}^\tau(u)$, or $O_{12}^+(u)$, the claim follows from Lemma \ref{l:zy}(4) and Lemma \ref{l:small}.

Let $S=S_8(u)$, $O_9(u)$, or $O_8^-(u)$. Since $t(r_z,S)=t(r_y,S)=4$, 
the description of cocliques of size 4 in Table \ref{tab:t4} and Lemma  \ref{l:vRy} imply that $R_z(L)\subseteq R_i(\tau u)$ and  $R_y(L)\subseteq R_j(\tau u)$, where $i,j\in \{8,6,4,3\}$.  If $i,j\in \{6,4,3\}$, then every prime dividing $u^2-1$ is adjacent to both $r_z$ and $r_y$, which is impossible. 

Similarly, if $S=L_7^\tau(u)$, then $R_z(L)\subseteq R_i(\tau u)$ and  $R_y(L)\subseteq R_j(\tau u)$, where $i,j\in \{7,6,5,4\}$. If  $i,j\in \{6,5,4\}$, then every prime in $\pi(u^2-1)\setminus\{7\}$ is adjacent to both $r_z$ and $r_y$ (cf. \cite[Fig. 6]{24GrPan.t}), which is not the case.
\end{proof}

\begin{lemma} \label{l:aut} Suppose that for every $s\in R_y(L)$, a Sylow  $s$-subgroup of $S$ is cyclic. Then $\pi(\overline G/S)\subseteq \pi(q(q^4-1))$.
\end{lemma}

\begin{proof}
Let $r_z\in R_z(L)$, $r_y\in R_y(L)$, $r\in  \pi(\overline G/S)$ and $r\not\in\pi(q^4-1)$. Then $rr_z, rr_y\not\in\omega(G)$. Since $S\neq L_2(u)$ and $r_z>3$, a Sylow $r_z$-subgroup of $S$ is cyclic (see, for example,  \cite[Lemma 2.3]{20YanGrVas}). By \cite[Lemma 1.7]{20GrZv.t}, the prime $r$ divides $r_z-1$ and $r_y-1$. It follows that $r$ divides $(k_z(L)-1, k_y(L)-1)$. It remains to apply Lemma \ref{l:igcd} and note that $\{2,3,5\}\subseteq\pi(q(q^4-1))$.
\end{proof}

Now we are ready to eliminate all possibilities for $S$. We will repeatedly use the following argument: if $k_i(\tau u)$ divides $a\in \omega(S)$ and $k_j(\tau u)$ divides $b\in \omega(S)$, where $i$ and $j$ are as in Lemma \ref{l:ij}, then $a$ divides $m_z(L)$ and $b$ divides $m_y(L)$.

\begin{lemma}\label{l:s8} $S$ is not one of $S_8(u)$, $O_9(u)$, $O_8^-(u)$.
 
\end{lemma}

\begin{proof}
Assume the opposite. By Lemma \ref{l:ij}, we have $R_x(L)\subseteq R_8(u)$ and $R_w(L)\subseteq R_i(u)$, where $\{w,x\}=\{z,y\}$ and  $i\in\{3,6,4\}$. In particular,  
\begin{equation}\label{e:s8ineq}
\frac{u^4+1}{2}\leq m_x(L)\quad\text{and}\quad   k_w(L)\leq \max\{k_3(u), k_6(u), k_4(u)\}\leq u^2+u+1.                                                                                                 
 \end{equation}

Let $L=L_8^\varepsilon(q)$. It is easy to see that $m_7(L)\leq (q^7+1)/2<7q^7/12$, and $m_8(L)<7q^7/12$ by Lemma \ref{l:bounds}. Hence  $m_x(L)\leq 7q^7/12$. Also $k_w(L)\geq (q^4+1)/2$. Combining these bounds with \eqref{e:s8ineq} yields $6u^{4}<7q^7$ and $q^4\leq 2u^2+2u+1$. Thus  $6u^4<7(2u^2+2u+1)^{7/4}$, whence $u\leq 17$ and then $q=5$. This contradicts  Lemma~\ref{l:small}.

Similarly, if $L=O_{10}^\varepsilon(q)$ or $O_{12}^+(q)$, then $m_x(L)\leq (q^4+1)(q+1)/2\leq q^5$ and $k_w(L)\geq q^4/6+1$, which yields $u^{4}<2q^5\leq 2(6(u^2+u))^{5/4}$. It follows that $u\leq 7$ и $q<5$, a contradiction. \end{proof}

\begin{prop}\label{p:l8}
If $L=L_8^\varepsilon(q)$, then $t(S)\neq 4$.
\end{prop}

\begin{proof}
Assume the opposite. By Lemma \ref{l:s8}, the group $S$ is not $S_8(u)$, $O_9(u)$, $O_8^-(u)$. Let  $i$ and  $j$  be as in Lemma \ref{l:ij}. 
Replacing $j$ by $i$ if necessary, we may assume that $R_7(\varepsilon q)\subseteq R_i(\tau u)$. 

Let $k=5,6$. By Lemma \ref{l:ij}, for every $r_y\in R_y(L)$, a Sylow $r_y$-subgroup of $S$ is cyclic and so $R_k(\varepsilon q)\cap \pi(\overline G/S)=\varnothing$ by Lemma \ref{l:aut}. Applying Lemma \ref{l:large}, we conclude that $R_k(\varepsilon q)\subseteq \pi(S)$ and $t(r,S)=4$ for every   $r\in R_k(\varepsilon q)$.

\textbf{1.} Let $S=L_7^\tau(u)$. Using the description of cocliques of size  $4$, we conclude that there is $l\in \{7,6,5,4\}$ such that $R_5(\varepsilon q)\subseteq R_l(\tau u)$. So $\{i,j,l\}\subseteq \{7,6,5,4\}$. It follows that there are  $w\in \{8,7,5\}$ and $t\in \{6,4\}$ such that  $R_w(\varepsilon q)\subseteq k_t(\tau u)$ and hence 
\begin{equation}\label{e:l8ineq1}
q^4/6+1< k_w(\varepsilon q)\leq k_t(\tau u)\leq u^2+u+1. 
\end{equation}

Since $7\in\{i,j\}$, there is $x\in\{z,y\}$ such that $k_7(\tau u)$ divides $m_x(L)$. As noted above, $m_x(L)\leq 7q^7/12$.  Also $k_7(\tau u)>u^6/8$ by  Lemma \ref{l:bounds}. Thus $u^6/8<7q^7/12$.  Combining this with  \eqref{e:l8ineq1} yields $3u^6<14q^7<14(6(u^2+u))^{7/4}$, whence $u\leq 7$ and $q<5$, a contradiction. 

\textbf{2.} Let $S=L_8^\tau(u)$ and $i=7$. Then $k_8(q)$ divides $k_8(u)$ and $m_8(u)$ divides  $m_8(q)$. It is easy to check that $k_8(q)\neq k_8(u)$. Prime divisors of $k_8(u)$ are at least $17$, so $q^4\leq 2u^4/17$.
Thus  $$\frac{u^7}{16}<\frac{u^8-1}{8(u+1)}\leq m_8(u)\leq m_8(q)\leq q^7\leq \left(\frac{2}{17}\right)^{7/4}u^7<\frac{u^7}{42},$$ a contradiction.

\textbf{3.} We are left with the case when $S$ is one of $L_8^\tau(u)$, $O_{10}^\tau(u)$, or $O_{12}^+(u)$ and $i\in\{5,8,10\}$. Let $m_i(S)$ be the only number in $\mu(S)$ divisible by $r_i(\tau u)$ (see Lemma \ref{l:zy} and Table \ref{tab:zy}). We will exploit the fact that $k_7(\varepsilon q)$ divides $k_i(\tau u)$ and the inequality 
\begin{equation}\label{e:l8ineq2}
 (u^4+1)(u-1)/4\leq m_i(S)\leq m_7(L)\leq (q^7+1)/2. 
\end{equation}

Suppose that $k_7(\varepsilon q)\neq k_i(\tau u)$. Prime divisors of  $k_i(\tau u)$ is at least $2i+1$ and so 
\begin{equation}\label{e:l8ineq3}
q^6/8<k_7(\varepsilon q)\leq k_i(\tau u)/11\leq 2u^4/11.                                                                          
\end{equation}
It follows from \eqref{e:l8ineq2} and \eqref{e:l8ineq3} that 
$(u^4+1)(u-1)<2(16u^4/11)^{7/6}+2$, whence  $u\leq 32$ and $q\leq 9$. But then $k_7(\varepsilon q)>q^6/2$ and the resulting inequality is $(u^4+1)(u+1)<2(4u^4/11)^{7/6}+2$, whence $u<2$. 

Thus  $k_7(\varepsilon q)=k_i(\tau u)$ and
\begin{equation}\label{e:l8eq}
\frac{q(q^6-1)}{q-\varepsilon}=(7,q-\varepsilon)k_i(\tau u)-1. 
\end{equation}
Denote  $(7,q-\varepsilon)k_i(\tau u)-1$ by $b$. Then $k_6(\varepsilon q)$ divides $b$. 

On the other hand, recall that $t(r,S)=4$ for every $r\in R_6(\varepsilon q)$. Also $r\not\in R_4(u)$ if $S=L_8^\tau(u)$ since otherwise $r$ is adjacent to primes in $R_8(u)$ and so to $r_z(L)$ or $r_y(L)$. The description of cocliques of size 4 implies that $R_6(\varepsilon q)\subseteq R$, where $R$ is as follows: $R=R_l(\tau u)$ with $l\in\{5,6\}$ if $S=L_8^\tau(u)$; $R=R_3(\tau u)$ or $R=R_6(\tau u)\cup R_4(\tau u)$ if $S=O_{10}^\tau(u)$;  and $R=R_8(u)$ or $R=R_6(u)\cup R_3(u)$ if $S=O_{12}^+(u)$. 

Assume that $(7, q-\varepsilon)=1$ and if $k_i(\tau u)=k_5(\epsilon u)$, then $(5,u-\epsilon)=1$. Then $b=k_i(\tau u)-1\in\{u^4, (u^4-1)/2, u(u^4-1)/(u-\epsilon)\}$. Since $R_6(\varepsilon q)\subseteq\pi(u(u^4-1))\cap R$, we have 
$S=O_{10}^\tau(u)$ and $R_6(\varepsilon q)\subseteq R_4(\tau u)$. By then 
$R_5(\varepsilon q)\subseteq R_3(\tau u)$, and hence $q^4/6<u^2+u$. Together with \eqref{e:l8ineq2}, this yields $(u^4+1)(u-1)<2(6(u^2+u))^{7/4}+2$, whence $u\leq 13$ and $q=5$, a contradiction. 

Now let $(q-\varepsilon, 7)=7$ or $k_i(\tau u)=(u^5-\epsilon)/(5(u-\epsilon))$. Suppose that $i=8$. Then $k_6(\varepsilon q)$ divides $7k_8(u)-1$. On the other hand, $k_6(\varepsilon q)$ divides one the numbers $\Phi_5(\tau u)$, $\Phi_3(\tau u)$ and $\Phi_6(\tau u)\Phi_4(u)$. Applying Lemma \ref{l:igcd1} and observing that prime divisors of $k_6(\varepsilon q)$ are congruent to 1 modulo 6, we conclude that $k_6(\varepsilon q)$ divides $31\cdot 61$, $151$, or $43\cdot 13$. It follows that $\varepsilon q\in\{7, -13\}$. Since $(q-\varepsilon, 7)=7$, we are left with the equation $k_8(u)=k_7(-13)$, which has no solutions in integers. 

Similarly, if $i=5$ or $10$, then $k_6(\varepsilon q)$ divides, on the one hand, one of the numbers $7\Phi_5(\pm u)-1$, $7\Phi_5(\pm u)-5$ и $\Phi_5(\pm u)-5$. On the other hand, it divides $\Phi_6(\pm u)\Phi_4(u)$, $\Phi_6(u)\Phi_3(u)$ or $\Phi_8(u)$. Observe that $\Phi_5(\pm u)\equiv 1\pmod {\Phi_4(u)}$ and hence $(b,\Phi_4(u))$ divides $6$, $2$ or $4$. This shows that $(k_6(\varepsilon q),\Phi_4(u))=1$. Thus $k_6(\varepsilon q)$ divides  $\Phi_6(u)\Phi_3(u)=\Phi_3(u^2)$ or $\Phi_8(u)=\Phi_4(u^2)$. Applying Lemma  \ref{l:igcd1}, we see that it divides $1297$, $337$,  $13\cdot 109$, or $19\cdot 43$ if $(q-\varepsilon, 7)=7$ and $97$ or $7\cdot 31$ otherwise.  It follows that $\varepsilon q=-25$ and so $k_5(\pm u)=k_7(-25)$. This equation has no solutions and the proof is complete. 
\end{proof}

\begin{prop}\label{p:o10}
If $L$ is one of the groups $O_{10}^\varepsilon(q)$ or $O_{12}^+(q)$, then $t(S)\neq 4$.
\end{prop}

\begin{proof}
Assume the opposite. By Lemma \ref{l:s8}, the group $S$ is not $S_8(u)$, $O_9(u)$, $O_8^-(u)$.

\textbf{1.} Let $S=L_7^\tau(u)$. By Lemma \ref{l:ij}, we have $R_x(L)\subseteq R_7(\tau u)$ and $R_w(L)\subseteq R_l(\tau u)$, where $\{x,w\}=\{z,y\}$ and $l\in\{ 6,5,4\}$. Hence $u^6/8<k_7(\tau u)\leq m_x(L)\leq (q^4+1)(q+1)/2$ and
$q^4/6+1<k_w(L)\leq (u^5-1)/(u-1)$. Solving this system of inequalities by excluding $q$, we conclude that $u\leq 37$.

If $(u-\tau,7)=1$, then we may write $k_7(\tau u)\geq 2u^6/3$ and then a stronger system yields $u\leq 8$ and $q\leq 11$. Evaluating $k_7(\tau u)$ for every $u\leq 8$ and then $e=e(k_7(\tau u),q)$ for every $7\leq q\leq 11$, we see that $e\geq 6$ and $e\not\in\{5,8,10\}$. This is a contradiction because  $k_7(\tau u)$ divides $m_x(L)$.

If $\tau u= -13, 27$, or $-29$, then $q\leq 43$. Arguing as in the previous paragraph, we conclude that $k_7(\tau u)$ cannot divide $m_x(L)$.

\textbf{2.} Let  $S=L_8^\tau(u)$ and choose $x\in\{z,y\}$ so that $R_x(L)\subseteq R_8(u)$. Then $$(u^8-1)/(8u-8)\leq m_8(u)\leq m_x(L)\leq (q^4+1)(q+1)/2$$ and $q^4/6+1<k_x(L)\leq k_8(u)\leq u^4+1$. Solving this system yields  $u\leq 5$ and $q<7$. 

\textbf{3.} Let $S=O_{10}^\tau(u)$. Again we choose $x\in\{z,y\}$ so that   $R_x(L)\subseteq R_8(u)$. Then $(u^4+1)(u-1)/4\leq m_8(u)\leq m_x(L)\leq (q^4+1)(q+1)/2$ and $k_x(\varepsilon q)$ divides $k_8(u)$.

Assume that $k_x(\varepsilon q)\neq k_8(u)$. Then $q^4/6+1<k_x(\varepsilon q)\leq k_8(u)/17\leq (u^4+1)/17$, which yields $u=3$ and $q<5$. 

Thus $k_x(\varepsilon q)=k_8(u)$, whence $x=5$ or $10$. More exactly,  $k_x(\varepsilon q)=k_5(\epsilon q)$, where $\epsilon=\varepsilon$ if  $L=O_{10}^\varepsilon(q)$ and $\epsilon\in\{+,-\}$ if $L=O_{12}^+(q)$. 
Subtracting 1 from both sides of the equation $k_5(\epsilon q)=k_8(u)$, we get  
\begin{equation}\label{e:o10eq}
\frac{q^5-\epsilon}{(q-\epsilon)(5,q-\epsilon)}-1=\frac{u^4+1}{(2,u-1)}-1. 
\end{equation}

Suppose that $(5,q-\epsilon)=1$. Then the left side of \eqref{e:o10eq}
is equal to $q(q^2+1)(q+\epsilon)$, while the right side is equal to $u^4$ or $(u^4-1)/2$. The first case is impossible, so $u$ is odd and $\{p\}\cup R_4(q)\subseteq \pi(u^4-1)$. Every odd prime dividing $u^2-1$ is adjacent either to primes in $R_i(\tau u)$ or to primes in $R_j(\tau u)$, where $\{i,j\}$ are as in Lemma \ref{l:ij}, and, therefore, to primes of $R_z(L)$ or $R_y(L)$. Hence  $\{p\}\cup R_4(q)\subseteq R_4(u)$. It follows that $q(q^2+1)$ divides $u^2+1$ and $u^2-1$ divides $2(q+\epsilon)$. Thus $q(q^2+1)\leq 2(q+\epsilon)+2$, a contradiction. 

Let now $(5,q-\epsilon)=5$. Then $k_5(\epsilon q)-1=f(\epsilon q)(\epsilon q-1)/5$, where $f(x)=x^3+2x^2+3x+4$. Note that $(f(\epsilon q))_5=5$ (see, for example, \cite[Lemma 2.8]{25Pan_arxiv}), and so the right side of  
\eqref{e:o10eq} cannot be $u^4$. Hence $u$ is odd and $f(\epsilon q)(\epsilon q-1)/5=(u^4-1)/2$.

Since $(u^4-1)/2\in \omega(S)$, we have $|f(\epsilon q)|\in \omega(L)$.
The number $f(\epsilon q)$ is coprime to $p$, and so it divides  $\exp_{p'}(L)$, which is equal to $\Phi_5(\varepsilon q)\cdot \prod_{l=1}^4\Phi_l(q^2)/2$ (see, for example, \cite[Lemma 2.6]{20GrZv.t}). It is clear that $(f(\epsilon q),\Phi_5(\epsilon q))=5$. By Lemma \ref{l:igcd2}, the number $(f(\epsilon q),\Phi_l(q^2))$, where $l\in \{1,2,3,4\}$, divides $10$, $4$, $7\cdot 31$, or $2\cdot 97$. Also $f(\epsilon q)\equiv 2\pmod 4$. Thus $f(\epsilon q)$ divides $a=10\cdot 7\cdot 31\cdot 97$, whence  $q\leq 59$. Now it is straightforward to show that $f(\epsilon q)$ does not divide $a$. 

\textbf{4.} Let $S=O_{12}^+(u)$ and $R_x(L)\subseteq R_5(u)$, 
$R_w(L)\subseteq R_{10}(u)$, where $\{x,w\}=\{z,y\}$. 

Assume that $k_x(q)\neq k_5(u)$ or $k_w(q)\neq k_{10}(u)$. Then
$q^4/6<(u^5-1)/(11u-11)$. Also $(u^5+1)/2\leq m_w(L)\leq (q^4+1)(q+1)/2$. 
Thus $$u^5+1<(q^4+1)(q+1)<\left(\frac{6}{11}\cdot \frac{u^5-1}{u-1}+1\right)\left(\left(\frac{6}{11}\cdot \frac{u^5-1}{u-1}\right)^{1/4}+1\right),$$ whence $u\leq 3$ and $q^4<66$, a contradiction. 

So $k_x(q)=k_5(u)$ and $k_w(q)=k_{10}(u)$. If $L=O_{12}^+(q)$, this contradicts Lemma \ref{l:k5}. Let $L=O_{10}^\varepsilon(q)$ and choose $\epsilon$ so that $k_8(q)=k_5(\epsilon u)$. Then $k_5(\varepsilon q)=k_5(-\epsilon u)$ and by Lemma \ref{l:k5}, it follows that $(5,q-\varepsilon)\neq (5, u+\varepsilon)$. If $(5,u-\epsilon)=5$, then $(5,q-\varepsilon)=5$ and we have $q^4/2<k_8(q)=k_5(\epsilon u)<2u^4/5$ and $2q^5/>k_5(\varepsilon q)=k_5(-\epsilon u)>u^4/2$, a contradiction. 

Hence  $(5,u-\epsilon)=1$. Then
$(q^4-1)/2=u(u^4-1)/(u-\epsilon)=u(u^2+1)(u+\epsilon)$. If $r_4(q)$ divides  $u+\epsilon$, then $r_4(q)$ is adjacent to every $r_{10}(\epsilon u)$ in $GK(S)$ and so to $r_5(\varepsilon q)$, but this is not the case.   It follows that $r_4(q)\in\{v\}\cup R_4(u)$ and therefore $r_4(q)$ is adjacent to all primes in $\pi(S)\setminus (R_5(u)\cup R_{10}(u))$. Now observe that  $r=r_3(\varepsilon q)$ is coprime to $|\overline G/S|$ by Lemma \ref{l:aut} and so $r\in \pi(S)$ by Lemma \ref{l:large}. Since $rr_4(q)\not\in\omega(G)$, we conclude that $r\in R_5(u)\cup R_{10}(u)$ and then  $r$ is adjacent to  $r_8(q)$ or $r_{5}(\varepsilon q)$. This contradiction completes the proof of the proposition.
\end{proof}

Theorem \ref{t:main} follows from Lemma \ref{l:large} and Propositions  \ref{p:t7}--\ref{p:t5},  \ref{p:l8} and \ref{p:o10}. As we mentioned in the introduction, Theorem \ref{t:cor} follows from Theorem \ref{t:main} and Lemma \ref{l:red}. 

As we explained in the introduction, Theorem \ref{t:t4} is a consequence of Theorem \ref{t:cor} and previous results. To get $h(L)$ and $\mathcal H(L)$  for a simple group $L$ from Theorem \ref{t:t4} explicitly, one can use Tables 1--6 and 10 in \cite{23Survey} in the following way. If  $L$ is one of $L_3^\pm(q)$, then  $h(L)$ and $\mathcal H(L)$ is given the corresponding colums of Tables 1 and 2. If $L$ is one of $S_8(q)$, $O_9(q)$, $O_8^-(q)$ and $L\neq S_8(7^m)$, then $h(L)=\infty$ by Table 10. If $L=S_8(7^m)$, then $h(L)=1$ by  \cite{25Gr.t}. In the remaining cases, $h(L)<\infty$ and one should find the table corresponding to the type of $L$, find the section containing the condition of the form  ``$n\geq n_0$'' and take $h(L)$ and the description of $\mathcal H(L)$ in this section. Let, for example, $L=O_{10}^-(q)$, where $q=p^m$.  By the second section of Table~6, if $p=7$ or $(q^5+1,4)=4$, then $h(L)=1$; otherwise, $h(L)$ is equal to the number of divisors of $2(m)_2$ and $\mathcal H(L)$ is described via the last column. The justification of this method is as follows. In the sections of Tables 1--6 containing the condition of the form ``$n\geq n_0$'', the two last columns provide, in fact, a formula for the number $|\mathcal H(L)\cap\mathcal{AS}(L)|$ and a description of the set $\mathcal H(L)\cap\mathcal{AS}(L)$. These formula and description are valid for all $n\geq 5$ and so, as soon as we have proved that $L$ satisfies $\mathcal H(L)\subseteq\mathcal{AS}(L)$, the last columns give the solution to the recognition problem (cf. \cite[Conjecture 3.11]{23Survey}).

\section*{Acknowledgments}

The work was supported by the Russian Science Foundation, project
24-11-00127, \\https://rscf.ru/en/project/24-11-00127/.


\end{document}